\documentclass[reqno,12pt]{amsart}
\newcommand{\be}{\begin{eqnarray}}
\newcommand{\ee}{\end{eqnarray}}
\newcommand{\ben}{\begin{eqnarray*}}
\newcommand{\een}{\end{eqnarray*}}

\newtheorem{thm}{Theorem}[section]
\newtheorem{cor}{Corollary}[section]

\newtheorem{lema}{Lemma}[section]
\newtheorem{prop}{Proposition}[section]
\usepackage{graphicx}
\usepackage{psfrag}
\usepackage{epsfig,color}

\numberwithin{equation}{section}

\begin{document}
\title[Uniqueness of positive radial solutions]
      {On the uniqueness of solutions of
            a semilinear equation  in an annulus  }\thanks{This research was supported by
        FONDECYT-1190102 for the first and second author, and
        FONDECYT- 1170665 for the  third author and by JSPS KAKENHI Grant Number
        19K03595 and 17H01095 for the fourth author.}
\author[C. Cort\'azar]{Carmen Cort\'azar}
\address{Departamento de Matem\'atica, Pontificia
        Universidad Cat\'olica de Chile,
        Casilla 306, Correo 22,
        Santiago, Chile.}
\email{\tt ccortaza@mat.uc.cl}
\author[M. Garc\'\i a-Huidobro]{Marta Garc\'{\i}a-Huidobro}
\address{Departamento de Matem\'atica, Pontificia
        Universidad Cat\'olica de Chile,
        Casilla 306, Correo 22,
        Santiago, Chile.}
\email{\tt mgarcia@mat.uc.cl}
\author[P. Herreros]{Pilar Herreros}
\address{Departamento de Matem\'atica, Pontificia
        Universidad Cat\'olica de Chile,
        Casilla 306, Correo 22,
        Santiago, Chile.}
\email{\tt pherrero@mat.uc.cl}
\author[S. Tanaka]{Satoshi Tanaka}
\address{Mathematical Institute, Tohoku University,
         Aoba 6--3, Aramaki, Aoba-ku, Sendai 980--8578, Japan.}
\email{\tt satoshi.tanaka.d4@tohoku.ac.jp}

%\maketitle

%%%%%%%%%%%%%%%%%%%%%%%%%%%%%%%%%%%%%%%%%%%%%%%%%%%%%%%%%%%%%%%%%%%%%%%%

\begin{abstract}
	We establish the  uniqueness of positive   radial  solutions
	of
	$$\begin{cases} \Delta u
	+f(u)=0,\quad x\in A \\ u(x) =0 \quad x\in \partial A
	\end{cases} \leqno( P) $$
	where $A:=A_{a,b}=\{ x\in {\mathbb R}^n : a<|x|<b \}$, $0<a<b\le\infty$.
	We assume that the nonlinearity  $f\in C[0,\infty)\cap C^1(0,\infty)$ is such that $f(0)=0$ and  satisfies some
	convexity and growth conditions, and either $f(s)>0$ for all $s>0$, or has one zero at $B>0$,
	is non positive and not identically 0 in $(0,B)$ and it is positive in $(B,\infty)$.
\end{abstract}

\maketitle
\today

\section{Introduction and main results}

In this paper we study the uniqueness of positive radial solutions of problem \eqref{P}
\begin{equation}
 \begin{cases}
  \Delta u +f(u)=0, \quad x\in A, \\
  u=0, \quad x\in \partial A,
 \end{cases}
  \label{P}
\end{equation}
where $n\ge2$, $A=\{ x\in {\mathbb R}^n : a<|x|<b \}$, $0<a<b\le\infty$ and $f$ is a smooth function.
When $b=\infty$, we mean that $A$ is the exterior of the ball centered at the origin with radius $a$ so the boundary condition is interpreted as $u=0$ on
$|x|=a$ and $\lim_{|x|\to\infty}u(x)=0$.
  We will  give conditions on $f$ under which
\begin{equation}\label{R}
 \begin{cases}
  u''(r)+\displaystyle\frac{n-1}{r}u'(r)+f(u(r))=0, \quad r \in (a,b),\quad n\ge 2, \\[1ex]
  u(a)=u(b)=0
 \end{cases}
\end{equation}
has at most one positive solution for every $a$ and $b$.

We will   assume that $f$ satisfies the following basic structural and subcriticality conditions:
\begin{enumerate}
	\item[$(f_1)$] $f$ is continuous in $[0,\infty)$, continuously differentiable in $(0,\infty)$ and $f'\in L^1(0,1)$,
	\item[$(f_2)$] $f(0)=0$ and either there exist $\beta> B> 0$ such that $f(s)>0$ for $s>B$,
	$f(s)\le 0$ for $s\in[0,B]$   with $f(s)<0$ in $(0,\epsilon)$ for some $\epsilon>0$ and  $\int_0^\beta f(t)dt=0$, or $\beta=B=0$ and $f(s)>0$ for $s>B$.
		\item[$(f_3)$] $\displaystyle\Bigl(\frac{F}{f}\Bigr)'(s)\ge \frac{n-2}{2n}$ for all $s>\beta$,  where $F(s)=\int_0^sf(t)dt$ ( subcriticality),
\end{enumerate}

Observe that if $b=\infty$, then in order to have existence of a solution of
\eqref{P} with $u(r)>0$ for $ r \in (a,\infty)$ it is necessary that $B>0$, so we will assume it whenever we treat the case $b=\infty$.

It is known that the number of nonequivalent nonradial solutions of
this problem tends to infinity as $a \to \infty$
for the case where $b=a+1$, $f(s)\equiv s^p$, $p>1$ and $p<(n+2)/(n-2)$
if $n\ge 3$.
See Byeon \cite{byeon}, Coffman \cite{coff2} and Li \cite{li}.
The uniqueness of radial solutions on an annulus has been studied in
\cite{ChG}, \cite{coff3}, \cite{fmt}, \cite{fl}, \cite{ko}, \cite{kz},
\cite{lz}, \cite{n}, \cite{nn}, \cite{sw1}, \cite{sw2}, \cite{ta},
and \cite{y2}.
%Cheng and Guang \cite{Chg} and Fu and Lin \cite{fl} gave the uniqueness
%results of positive solutions of \eqref{R},
The case $f(s)\equiv s^p-s$, $p>1$ is treated by
Coffman \cite{coff3}, Yadava \cite{y1}, \cite{y2}, Tang \cite{ta}
and Felmer, Martinez and Tanaka \cite{fmt}.
More general equations are considered by Erbe and Tang \cite{et2} and
Li and Zhou \cite{lz}.
Kwong and Zhang \cite{kz} considered the boundary condition $u'(a)=0$ and
$u(b)=0$.
Shioji and Watanabe \cite{sw1}, \cite{sw2} introduced a generalized
Poho\v{z}aev identity and presented the uniqueness results of a
positive radial solution in a ball, the entire space, an annulus, or an
exterior domain under Dirichlet boundary condition.
The following Theorems A and B have been established by
Ni and Nussbaum \cite[Theorem 1.8]{nn}.

\medskip

\noindent{\bf Theorem A (Ni and Nussbaum \cite[Theorem 1.8]{nn}).} \it
Let $n \ge 3$ and $b<\infty$.
Assume that $f \in C^1[0,\infty)$, $f(s)>0$ for $s>0$ and
\begin{equation*}
1 < \frac{sf'(s)}{f(s)} \leq \frac{n}{n-2} + \frac{2}{(b/a)^{n-2}-1}
\quad \mbox{for} \ s>0.
\end{equation*}
Then problem \eqref{R} has at most one positive solution.

\rm

\medskip

From Theorem A, we have the following corollary.

\medskip

\noindent{\bf Corollary A.} \it
Let $n \ge 3$.
Assume that $f \in C^1[0,\infty)$, $f(s)>0$ for $s>0$ and
\begin{equation*}
1 < \frac{sf'(s)}{f(s)} \leq \frac{n}{n-2} \quad \mbox{for} \ s>0.
\end{equation*}
Then problem \eqref{R} has at most one positive solution for every
$a$ and $b$ with $0<a<b<\infty$.
In particular, if $f(s)=s^p+s^q$, $1<q \le p\le n/(n-2)$, then
problem \eqref{R} has at most one positive solution for every
$a$ and $b$ with $0<a<b<\infty$.

\rm

\medskip

When $n=2$, recently, Shioji, Tanaka and Watanabe \cite{stw} proved that
problem \eqref{R} has at most one positive solution for every $a$ and
$b$ with $0<a<b<\infty$, provided $n=2$, $f \in C^1[0,\infty)$,
$f(s)>0$ and $1<sf'(s)/f(s) \le 5$ for $s>0$.
Therefore, a positive solution is unique if $sf'(s)/f(s)$ is small in
some sense.
If not, problem \eqref{R} can have at least three positive solutions.
Indeed, Ni and Nussbaum \cite[Theorem 1.10]{nn} proved the following result.

\medskip

\noindent{\bf Theorem B (Ni and Nussbaum \cite[Theorem 1.10]{nn}).} \it
Let $n \ge 3$.
Assume that $f(s)\equiv s^p+s^q$ for $s>0$, $1<p<(n+2)/(n-2)<q<\infty$.
Then there exist $a$ and $b$ with $0<a<b<\infty$ such that
problem \eqref{R} has at least three positive solutions.

\rm

\medskip

Actually, Ni and Nussbaum \cite{nn} considered the case where
$f(s)\equiv s^p+\varepsilon s^q$, $\varepsilon>0$, but the transformation
$v(r)=\varepsilon^\frac{1}{q-p}u(\varepsilon^\frac{p-1}{2(q-p)}r)$
reduces to the case $\varepsilon=1$.

Ni and Nussbaum \cite{nn} also proved that problem \eqref{R} has at
most one positive solution if $b/a$ is small in some sense when $B=0$,
and later Korman \cite[Theorem 3.3]{ko} improved it.

Yadava \cite{y1} employed the generalized Pohozaev identity and obtained
uniqueness results for the case $f(u)=u^p+u^q$, $p>q\ge 1$.

\medskip

\noindent{\bf Theorem C (Yadava \cite[Theorem 1.10]{y1}).} \it
Let $n \ge 3$.
Assume that $f(s)\equiv s^p+s^q$, $p>q\ge 1$.
Then problem \eqref{R} with $b<\infty$ has at most one positive solution
for every $a$ and $b$ with $0<a<b<\infty$ if one of the following conditions
{\rm (i)--(iii)} is satisfied{\rm :}
\begin{enumerate}
	\item[(i)] $q=1$ and $1<p\le (n+2)/(n-2)${\rm ;}
	\item[(ii)] $1<q<p<(n+2)/(n-2)$ and $(p-1)/(q+1)\le2/n${\rm ;}
	\item[(iii)] $(n+2)/(n-2)<q \le p$ and $(q-1)/(p+1)\ge2/n$.
\end{enumerate}

\rm

Our results for general $f$ improve the existing ones  in the sense that our conditions do not depend on the width of the annulus.

Our main results are the following.

\begin{thm}\label{main0}
 If $f$ satisfies $(f_1)$--$(f_3)$, and assume furthermore that $f$ satisfies
 \begin{enumerate}
 	\item[$(f_4)$] $ f(s)\leq f'(s) (s-B)$ for all $s>B $, $f(s)\not\equiv f'(s) (s-B)$ for $s\in [B,\beta]$ (superlinearity).
 	 \end{enumerate}
 Then problem \eqref{R}
 has at most one positive solution for every $a$ and $b$ with $0<a<b<\infty$.
\end{thm}

If $b=\infty$, we do not need the superlinearity. So in this case we have

\begin{thm}\label{2}
 If $f$ satisfies $(f_1)$-$(f_3)$, with $B>0$,  then problem \eqref{R} with
 $b=\infty$ has at most one positive solution for every $a$ with
 $0<a<\infty$.
\end{thm}

The proofs of our main results will follow after carefully comparing
two such positive solutions, assuming they exist.
We will use functionals found in \cite{et1}, \cite{fls} and \cite{cghy}.

\medskip

In the particular case  $f(s)=s^p+s^q$ from Theorem \ref{main0}  we obtain
the following corollary, which improves (ii) of Theorem C.

\begin{cor}\label{cor1}
 Let $n \ge 2$.
 Assume that $f(s)\equiv s^p+s^q$, where $p>1$ and $p>q>0$.
 Then problem \eqref{R} with $b<\infty$ has at most one positive solution
 for every $a$ and $b$ with $0<a<b<\infty$ if one of the following conditions
 {\rm (i)--(iv)} is satisfied{\rm :}
 \begin{enumerate}
  \item[(i)] $n\ge 6$ and $1\le q<p\le (n+2)/(n-2)${\rm ;}
  \item[(ii)] $2<n<6$ and $4/(n-2)\le q<p\le (n+2)/(n-2)${\rm ;}
  \item[(iii)] $2<n<6$, $0<q<4/(n-2)$ and $q<p\le P(q)$,
	       where
	       \begin{equation*}
	        P(q):=\frac{2(q+1)+\sqrt{(n+2)(q+1)(n+2-(n-2)q)}}{n}{\rm ;}	
	       \end{equation*}
  \item[(iv)] $n=2$ and $0< q<p \le q+1+2\sqrt{q+1}$.
 \end{enumerate}
\end{cor}

For existence of solutions in any annulus when $q>1$ we refer for example to \cite[Theorem B(viii)]{garai} and when $0<q\le 1<p$ to
\cite[Theorem B(vii)]{garai}, which gives existence if $b-a<C$ for some constant $C$ (and nonexistence if $b-a>C$).

Finally, the following result is an application of our results to the special case $f(s)=s^p-s^q$.

\begin{cor}\label{cor2}
 Let $n\ge 2$ and assume that $f(s)\equiv s^p-s^q$, $p>q>0$.
 Then problem \eqref{R} with $b=\infty$ has at most one positive solution
 for every $a \in (0,\infty)$ if one of the following conditions
 {\rm (i)--(iii)} is satisfied{\rm :}
 \begin{enumerate}
  \item[(i)] $n>2$, $0<q\le 4/(n-2)$ and $q<p\le (n+2)/(n-2)${\rm ;}
  \item[(ii)] $n>2$ and
	       \begin{equation*}
                \frac{4}{n-2} <q<\frac{4+\sqrt{2n(n+2)}}{2(n-2)}, \quad
	        q<p\le P(q),
	       \end{equation*}
	      where $P(q)$ is defined as in Corollary \ref{cor1}{\rm ;}
  \item[(iii)] $n=2$ and $0<q<p$.
 \end{enumerate}
 Moreover, if $p>1$ and one of {\rm (i)--(iii)} is satisfied, then
 problem \eqref{R} with $b<\infty$ has at most one positive solution
 for every $a$ and $b$ with $0<a<b<\infty$.

\end{cor}

\noindent{\em Remark.} It is worth noting that
\begin{equation*}
 P(q)=\frac{2(q+1)+\sqrt{(n+2)(q+1)(n+2-(n-2)q)}}{n}
\end{equation*}
is increasing with respect to $q \in (0,4/(n-2))$ and
is decreasing with respect to $q \in (4/(n-2),(4+\sqrt{2n(n+2)})/(2(n-2)))$.
Hence we have
\begin{equation*}
 P(q) \ge \frac{n+4}{2} \quad \mbox{for} \ q \in (0,4/(n-2))
\end{equation*}
and
\begin{equation*}
 P(q) \ge \frac{4+\sqrt{2n(n+2)}}{2(n-2)} \quad \mbox{for} \
 q \in \left(\frac{4}{n-2},\frac{4+\sqrt{2n(n+2)}}{2(n-2)}\right).
\end{equation*}
For existence of solutions in any annulus in this case we refer to \cite[Theorem B(iv)]{garai}.

\bigskip

Our paper is organized as follows. In Section \ref{prel} we prove some basic properties of the solutions, then in Sections 3 we prove some comparison results that are used to prove first Theorem \ref{main0} and in Section 4 we prove Theorem \ref{2}.
Finally, in Section \ref{proofofcor} we give the proofs of Corollaries \ref{cor1}--\ref{cor2}.

\section{Preliminaries}\label{prel}

The aim of this section is to establish several properties of positive
solutions of \eqref{R}.

\begin{prop}\label{1}
 Assume that $f$ satisfies $(f_1)$ and $(f_2)$.
 Let $u$ be a positive solution of \eqref{R}.
 Then the following  hold{\rm :}
 \begin{enumerate}
  \item[(i)] $(u'(r))^2+2F(u(r))>0$ for $r \in [a,b)${\rm ;}
  \item[(ii)] There exists a number $c \in (a,b)$ such that $u'(r)>0 $
	      for $r \in [a,c)$ and $u'(r)<0$ for $r \in (c,b)$.
	      Moreover, $u(c)>\beta$.
 \end{enumerate}
\end{prop}

\begin{proof}
 Set $I(r):=(u'(r))^2+2F(u(r))$.
 Then
 \begin{equation}
  I'(r) = - \frac{2(n-1)}{r} (u'(r))^2 \le 0, \quad r \in (a,b).
   \label{I'}
 \end{equation}
 Therefore $I(r)$ is nonincreasing on $(a,b)$, which implies that
 \begin{equation*}
  I(r) \ge I(b) \ge 2F(u(b)) = 2F(0) = 0, \quad r \in [a,b).
 \end{equation*}
 Now we will prove (i).
 Assume to the contrary that $I(r_1)=0$ for some $r_1 \in [a,b)$.
 Then $I(r)=0$ for $r \in [r_1,b)$ and hence $I'(r)=0$ for $r \in (r_1,b)$.
 From \eqref{I'} it follows that $u'(r)=0$ for $r \in (r_1,b)$, which
 means that $u(r)=u(b)=0$ for $r \in (r_1,b)$.
 This contradicts the fact that $u(r)>0$ for $r \in (a,b)$.
 Consequently, (i) follows.

 \noindent (ii) If $u'(r_0)=0$ we have $0<I(r_0)=F(u(r_0))$, which implies that
 $u(r_0)>\beta\ge B$.
 From the equation in \eqref{R}, it follows that $u''(r_0)=-f(u(r_0))<0$.

Since $u(r)>0$ for $r \in (a,b)$ and $u(a)=u(b)=0$, there exists a
 number $c \in (a,b)$ such that $u'(r)>0$ for $r \in [a,c)$ and $u'(c)=0$.
 Now we claim that $u'(r) \ne 0$ for $r \in (c,b)$.
 Assume that $u'(d)=0$ for some $d \in (c,b)$.
 Then $u''(c)<0$ and $u''(d)<0$, which means that there exists
 $r_1 \in (c,d)$ such that $u'(r_1)=0$ and $u''(r_1) \ge 0$.
 This is a contradiction.
 Hence $u'(r) \ne 0$ for $r \in (c,b)$ as claimed.
 Since $u(c)>0$ and $u(b)=0$, we have $u'(r)<0$ for $r \in (c,b)$.
Finally, we observe that as $u'(c)=0$, then from (i) $0<I(c)=F(u(c))$, which implies that
$u(c)>\beta\ge B$.
%From the equation in \eqref{R}, it follows that $u''(r_0)=-f(u(r_0))<0$.We have already %shown that $u(c)>\beta$ by $u'(c)=0$.
 Therefore (ii) follows.
\end{proof}

\begin{prop}\label{u'->0}
	Assume that $f(s)$ satisfies $(f_1)$, $(f_2)$ and $b=\infty$.
	Let $u$ be a positive solution of \eqref{R}.
	Then
	
	\begin{enumerate}
	\item [(i)]
	\begin{equation*}
	\lim_{r\to\infty} r^{n-1}u'(r) =L,
	\end{equation*}
	for some $L \in (-\infty,0]$.
	\item [(ii)]
	\begin{equation*}
	\lim_{r\to\infty} ru'(r) =0.
	\end{equation*}
\end{enumerate}	
\end{prop}

\begin{proof}
(i)
	Since $u(r) \to 0$ as $r \to \infty$, then $B>0$.  By condition $(f_2)$, there
	exists $r_0>a$ such that $f(u(r))<0$ for $r \ge r_0$.
	Since $(r^{n-1}u')'=-r^{n-1}f(u)$, we find that $r^{n-1}u'(r)$
	is increasing for $r \ge r_0$, thus it converges to some
	$L \in (-\infty,\infty]$ as $r\to\infty$.
	If $L\in(0,\infty]$, then $r^{n-1}u'(r)$ is eventually positive,
	and hence $u(r)$ is eventually increasing.
	This contradicts the fact that $u(r)>0$ on $(a,\infty)$ and
	$u(r)\to0$ as $r\to\infty$.
	Therefore, $L \in (-\infty,0]$.\\
	
\noindent (ii) 	
	For $n\geq3$ we have $$\lim_{r\to\infty} ru'(r) =\lim_{r\to\infty}r^{-(n-2)} r^{n-1}u'(r)=0\cdot L=0.$$
	For $n=2$ we proved above that   $$\lim_{r\to\infty} ru'(r)=L\in (-\infty,0].$$
	Suppose $L<0$. Then there exists $r_0>a$ such that, for $r>r_0$, $ru'(r)<L/2$. Hence
	
	$$u(r)= u(r_0)+ \int_{r_0}^r u'(t)\,dt<u(r_0)+\frac{ L}{2}\int_{r_0}^r \frac{1}{t}\,dt =u(r_0)+\frac{ L}{2}(\ln(r)-\ln(r_0)).$$ We can choose $r$ sufficiently large so $u(r)<0$, a contradiction, and thus $L=0$ when $n=2$.

\end{proof}

\section{Proof of Theorem \ref{main0}}

To prove Theorem \ref{main0}, we assume that \eqref{R}, with $b<\infty$,  has two distinct
positive solutions $u_1$ and $u_2$ with  $0<u_1'(a)<u_2'(a)$.
Comparing two such solutions carefully, as $r$ goes  from $a$ to $b$, will lead to
a contradiction, proving the theorem.

 We will assume that  $f$ satisfies $(f_1)$--$(f_3)$ for  Steps one through five.  For Step six and the Final Step we will assume also $(f_4)$, the superlinearity, to guarantee the existence of an
intersection point.

By Proposition \ref{1}, there exists $c_i$ such that
$u_i'(r)>0$ for $r \in [a,c_i)$ and $u_i'(r)<0$ for $r \in (c_i,b)$.
Denote by $r_i(s)$ and $\bar r_i(s)$ inverse functions of $u_i(r)$ on
$[a,c_i)$ and $(c_i,b)$, respectively.
Define
\begin{equation*}
 M_i:=u_i(c_i), \quad i=1,2.
\end{equation*}
and remember, by Proposition \ref{1}, that
\begin{equation*}
 M_i>\beta, \quad i=1,2.
\end{equation*}
We will use the following relation:
\begin{equation*}
 r_i'(s) = \frac{1}{u_i'(r_i(s))}, \quad
 \bar r_i'(s) = \frac{1}{u_i'(\bar r_i(s))}, \quad i=1,2,
\end{equation*}
where $'$ represents the derivative of the function with respect to its natural variable, that is, $r_i'$ stands for the derivative of $r_i$ with respect to $s$, and $u_i'$ denotes the derivative of $u_i$ with respect to $r$. Then
\begin{equation*}
 \frac{1}{r_i'(M_i)} = u_i'(r_i(M_i)) = u_i'(c_i) = 0
\end{equation*}
and similarly
\begin{equation*}
 \frac{1}{\bar r_i'(M_i)} = 0.
\end{equation*}

For our first step we will use the functional
$$
 V_i(s) := (r_i(s))^{2(n-1)} \Bigl( (u_i'(r_i(s)))^2 + 2F(s) \Bigr).
 \quad i=1,2.
$$
The functional $\sqrt{V_i(s)}$ was used first in \cite{cghy}.

\begin{prop} \label{subida}{\bf (Step one)}
 Assume that $f$ satisfies $(f_1)$ and $(f_2)$.
 Let $u_1$ and $u_2$ be two positive solutions of
 \eqref{R} with $0<u_1'(a)<u_2'(a)$. Then
 \begin{equation}
 \frac{(r_1(\beta))^{n-1}}{r_1'(\beta)} <
 \frac{(r_2(\beta))^{n-1}}{r_2'(\beta)},
 \label{atbeta}
 \end{equation}
 and if $\beta >0$, then  $r_1(s)> r_2(s)$ for $s \in (0,\beta]$.
\end{prop}
 To prove Proposition \ref{subida} we will first prove the following lemma.

 \begin{lema}\label{V<V}
 Assume that $f$ satisfies $(f_1)$ and $(f_2)$.
 Suppose that $\beta>0$ and let $u_1$ and $u_2$ be two positive solutions of
 \eqref{R} with $0<u_1'(a)<u_2'(a)$.
 If there exists $s_0 \in (0,\min\{M_1,M_2\})$ such that
 $V_1(s)\le V_2(s)$ for $s \in [0,s_0]$, then $r_1(s)>r_2(s)$ and
 $r_1'(s)>r_2'(s)$ for $s \in (0,s_0]$.
\end{lema}
\begin{proof}
 Note that $V_1(0)=a^{2(n-1)}(u_1'(a))^2<a^{2(n-1)}(u_2'(a))^2=V_2(0)$,
$r_1(0)=a=r_2(0)$ and
\begin{equation*}
r_1'(0) = \frac{1}{u_1'(a)} > \frac{1}{u_2'(a)} = r_2'(0).
\end{equation*}
Then there exists $\delta>0$ such that $r_1(s)>r_2(s)$ and
$r_1'(s)>r_2'(s)$ for $s \in (0,\delta)$.
Assume that there exists $s_1 \in (0,s_0]$ such that
$r_1(s)>r_2(s)$ and $r_1'(s)>r_2'(s)$ for $s \in (0,s_1)$ and either
$r_1(s_1)=r_2(s_1)$ or $r_1'(s_1)=r_2'(s_1)$.
Since $r_1'(s)>r_2'(s)$ for $s \in (0,s_1)$, we conclude that
$r_1(s)-r_2(s)$ is strictly increasing in $(0,s_1)$.
Hence, $r_1(s_1)-r_2(s_1)>r_1(0)-r_2(0)=a-a=0$, that is, $r_1(s_1)>r_2(s_1)$.
Therefore, $r_1'(s_1)=r_2'(s_1)$, which implies that
$u_1'(r_1(s_1))=u_2'(r_2(s_1))$.
We observe that
\begin{align*}
V_1(s_1)
& = (r_1(s_1))^{2(n-1)} \Bigl( (u_1'(r_1(s_1)))^2 + 2F(s_1) \Bigr) \\
& > (r_2(s_1))^{2(n-1)} \Bigl( (u_2'(r_2(s_1)))^2 + 2F(s_1) \Bigr)
= V_2(s_1).
\end{align*}
This is a contradiction.
Consequently, $r_1(s)>r_2(s)$ and $r_1'(s)>r_2'(s)$ for $s \in (0,s_0]$.
\end{proof}

\begin{proof} [{\bf Proof of Step one}]

Observe that if $\beta=0$ the proposition is trivially true. So from now on $\beta>0$.
 First we will prove that $V_1(s)<V_2(s)$ for $s \in [0,\beta]$.
 Since
 $$V_1(0)=a^{2(n-1)}(u_1'(a))^2<a^{2(n-1)}(u_2'(a))^2=V_2(0),$$
 there exists $\delta>0$ such that $V_1(s)<V_2(s)$ for $s \in [0,\delta)$.
 Assume that there exists $s_0 \in (0,\beta]$ such that $V_1(s_0)=V_2(s_0)$
 and $V_1(s)<V_2(s)$ for $s \in [0,s_0)$.
 From Lemma \ref{V<V} it follows that $r_1(s)>r_2(s)$ and
 $r_1'(s)>r_2'(s)$ for $s \in (0,s_0]$.
 We observe that
 \begin{equation*}
  V_i'(s) = 4(n-1) (r_i(s))^{2n-3} r_i'(s)F(s).
 \end{equation*}
 Since $F(s)<0$ for $s \in (0,s_0)$, we have $V_1'(s)<V_2'(s)$ for
 $s \in (0,s_0)$, which implies that $V_2(s)-V_1(s)$ is strictly
 increasing in $(0,s_0)$.
 Then $V_2(s_0)-V_1(s_0)>V_2(0)-V_1(0)>0$.
 This contradicts the fact that $V_1(s_0)=V_2(s_0)$.
 Therefore, $V_1(s)<V_2(s)$ for $s \in [0,\beta]$.
 In particular, $V_1(\beta)<V_2(\beta)$, which implies \eqref{atbeta}.
 Using Lemma \ref{V<V} again, we conclude that $r_1(s)>r_2(s)$ for
 $s \in (0,\beta]$.
\end{proof}
\begin{prop}\label{subida2} {\bf (Step two)}
 Assume that $f$ satisfies $(f_1)$ and $(f_2)$.
 Let $u_1$ and $u_2$ be two positive solutions of \eqref{R}
 with $0<u_1'(a)<u_2'(a)$.
 Then $M_1<M_2$ and
 \begin{equation}
  r_1(s)>r_2(s) \quad \mbox{and} \quad
  \frac{(r_1(s))^{n-1}}{r_1'(s)} < \frac{(r_2(s))^{n-1}}{r_2'(s)}
   \label{r1<r2}
 \end{equation}
 for all $s\in(\beta,M_1)$.
\end{prop}

\begin{proof}
 From Proposition \ref{subida} we have that the second part of  \eqref{r1<r2} holds at $s=\beta$, and $r_1(s)>r_2(s)$ in some right neighborhood of
 $\beta$ when $\beta\ge 0$.
 Hence $w(s):=(r_1(s))^{n-1}u_1'(r_1(s))-(r_2(s))^{n-1}u_2'(r_2(s))<0$
 and $r_1(s)>r_2(s)$ in some right neighborhood of $\beta$.
 Let $M:=\min\{M_1,M_2\}$.
 Assume that there exists $s_0 \in (\beta,M)$ such that $w(s)<0$ and
 $r_1(s)>r_2(s)$ for $s \in [\beta,s_0)$ and either $w(s_0)=0$ or
 $r_1(s_0)=r_2(s_0)$.
 Then
 \begin{equation*}
  r_1'(s) > \Bigl(\frac{r_1(s)}{r_2(s)} \Bigr)^{n-1} r_2'(s) > r_2'(s),
  \quad s \in [\beta,s_0).
 \end{equation*}
 Hence $r_1(s)-r_2(s)$ is strictly increasing on $[\beta,s_0)$.
 Since $r_1(\beta)-r_2(\beta) \ge 0$, we have $r_1(s_0)-r_2(s_0)>0$.
 Therefore, $w(s_0)=0$ and $r_1(s_0)>r_2(s_0)$.
 Then $w'(s_0) \ge 0$.
 On the other hand, since
 \begin{align*}
  w'(s) & = -f(s) \Bigl(
            \frac{(r_1(s))^{n-1}}{u_1'(r_1(s))}
          - \frac{(r_2(s))^{n-1}}{u_2'(r_2(s))}
                 \Bigr) \\
  & = - \frac{f(s)}{(r_2(s))^{n-1}u_2'(r_2(s))}
        \Biggl( - \frac{(r_1(s))^{n-1}w(s)}{u_1'(r_1(s))}
                 + (r_1(s))^{2(n-1)} - (r_2(s))^{2(n-1)} \Biggr),
 \end{align*}
 we have
 \begin{align*}
  w'(s_0) = - \frac{f(s_0)}{(r_2(s_0))^{n-1}u_2'(r_2(s_0))}
        \Biggl( (r_1(s_0))^{2(n-1)} - (r_2(s_0))^{2(n-1)} \Biggr) < 0.
 \end{align*}
 This is a contradiction.
 Therefore, $w(s)<0$ and $r_1(s)>r_2(s)$ for $s \in (\beta,M)$.
 This shows that \eqref{r1<r2} holds for $s \in (\beta,M)$.

 Finally we prove that $M=M_1$.
 Assume that $M_2<M_1$.
 Letting $r \to M_2$ in $w(s)<0$, we have
 \begin{equation*}
  0 \ge \lim_{s\to M_2} w(s) = (r_1(M_2))^{n-1}u_1'(r_1(M_2)) > 0,
 \end{equation*}
 which is a contradiction.
\end{proof}

For our next steps
we will use the functional introduced first by Erbe and Tang in
\cite{et1}, and  we shall use many of their ideas. Set
$$
 P_i(s) = -2n\frac{F(s)}{f(s)}\frac{(r_i(s))^{n-1}}{r_i'(s)}
        -\frac{(r_i(s))^n}{(r_i'(s))^2}
	-2(r_i(s))^nF(s), \quad i=1,2.
$$
We observe that
$$
 P_i(s) = -2n\frac{F(s)}{f(s)}\frac{(r_i(s))^{n-1}}{r_i'(s)}
          -(r_i(s))^{-(n-2)}V_i(s), \quad i=1,2.
$$
and
\begin{equation}
 P_i'(s) = \Bigl(n-2-2n\Bigl(\frac{F}{f}\Bigr)'(s)\Bigr)
           \frac{(r_i(s))^{n-1}}{r_i'(s)}, \quad i=1,2.
\end{equation}

\begin{prop}\label{lem31}{\bf (Step three)}
 Assume that $f$ satisfies $(f_1)$--$(f_3)$.
 Let $u_1$ and $u_2$ be positive solutions of \eqref{R} with
 $0<u_1'(a)<u_2'(a)$.
 Then
 $$
  P_1(M_1)>P_2(M_1).
 $$
\end{prop}

\begin{proof}
 If $\beta>0$, then Proposition \ref{subida} implies that
 \begin{equation*}
  r_1(\beta)>r_2(\beta), \quad
    \frac{(r_1(\beta))^{n-1}}{r_1'(\beta)}
  < \frac{(r_2(\beta))^{n-1}}{r_2'(\beta)},
 \end{equation*}
 and hence
 \begin{align*}
  P_1(\beta)-P_2(\beta)
  & = - \frac{(r_1(\beta))^n}{(r_1'(\beta))^2}
      + \frac{(r_2(\beta))^n}{(r_2'(\beta))^2} \\
  & = - \Bigl(\frac{(r_1(\beta))^{n-1}}{r_1'(\beta)}\Bigr)^2
        \frac{1}{(r_1(\beta))^{n-2}}
      + \Bigl(\frac{(r_2(\beta))^{n-1}}{r_2'(\beta)}\Bigr)^2
        \frac{1}{(r_2(\beta))^{n-2}} > 0.
 \end{align*}
 For the case $\beta=B=0$, integration by parts gives
 \begin{equation*}
  F(s) = sf(s) -\int_0^stf'(t)dt,
 \end{equation*}
and by L'H\^opital's rule we have
 \begin{equation*}
  \lim_{s\to0} \frac{F(s)}{f(s)}
   = \lim_{s\to0} s - \lim_{s\to0} \frac{\int_0^s tf'(t) dt}{f(s)}
   = 0 - \lim_{s\to0} \frac{sf'(s)}{f'(s)}
   = 0
 \end{equation*}
 thus
 \begin{eqnarray*}
  P_1(\beta)-P_2(\beta)
  = \lim_{s\to0} (P_1(s)-P_2(s))
 & = &-\frac{a^n}{(r_1'(0))^2}+\frac{a^n}{(r_2'(0))^2} \\
  &= & a^n (-(u_1'(a))^2+(u_2'(a))^2) > 0.
 \end{eqnarray*}
 Using $(f_3)$ and Proposition \ref{subida2}, we have
 $$
  \left(P_1-P_2\right)'(s)
  = \Bigl(n-2-2n\Bigl(\frac{F}{f}\Bigr)'(s)\Bigr)
    \Bigl( \frac{(r_1(s))^{n-1}}{r_1'(s)}
         - \frac{(r_2(s))^{n-1}}{r_1'(s)} \Bigr)
  \ge  0
 $$
 for $s \in (\beta,M_1)$.
 Thus $P_1(s)-P_2(s)\ge P_1(\beta)-P_2(\beta)>0$ which implies that
 $P_1(s)>P_2(s)$ for $s\in [\beta, M_1]$.
\end{proof}

We will now study the solutions going down from their maximum points, and their corresponding inverses $\bar r_i$.

 For our next step we will use $\bar P_i(s)$ the corresponding functional to $P_i(s)$.
\begin{equation}
 \begin{gathered}
  \bar P_i(s)
   = - 2n\frac{F(s)}{f(s)}\frac{(\bar r_i(s))^{n-1}}{\bar r_i'(s)}
     - \frac{(\bar r_i(s))^n}{(\bar r_i'(s))^2}
     - 2(\bar r_i(s))^nF(s), \\
  \bar P_i'(s)
    = \Bigl(n-2-2n\Bigl(\frac{F}{f}\Bigr)'(s)\Bigr)
      \frac{(\bar r_i(s))^{n-1}}{\bar r_i'(s)}.
 \end{gathered}
\end{equation}

\begin{prop}\label{lem32}{\bf (Step four)}
 Assume that $f$ satisfies $(f_1)$--$(f_3)$.
 Let $u_1$ and $u_2$ be positive solutions of \eqref{R} with
 $0<u_1'(a)<u_2'(a)$. If $\bar r_1(M_1)<\bar r_2(M_1)$ and $\bar r_1(s)$, $\bar r_2(s)$
 intersect each other in $[\beta,M_1)$, then

 $$
  \frac{(\bar r_1(s))^{n-1}}{\bar r_1'(s)} >
  \frac{(\bar r_2(s))^{n-1}}{\bar r_2'(s)}, \ s \in [s_I,M_1),
  \quad \mbox{and} \quad
  \bar P_1(s)>\bar P_2(s), \ s \in [s_I,M_1]
 $$
 where $s_I \in [\beta,M_1)$ is the largest intersection point of
 $\bar r_1(s)$ and $\bar r_2(s)$.
\end{prop}

\begin{proof}

 Since
 \begin{equation*}
  \lim_{s\to M_1} \frac{(\bar r_1(s))^{n-1}}{\bar r_1'(s)} = 0 >
  \frac{(\bar r_2(M_1))^{n-1}}{\bar r_2'(M_1)},
 \end{equation*}
 we find that
 $(\bar r_1(s))^{n-1}/\bar r_1'(s)>(\bar r_2(s))^{n-1}/\bar r_2'(s)$ in
 some left neighborhood of $M_1$.
 Assume that there exists $s_0 \in (s_I,M_1)$ such that
 \begin{equation*}
  \frac{(\bar r_1(s_0))^{n-1}}{\bar r_1'(s_0)} =
  \frac{(\bar r_2(s_0))^{n-1}}{\bar r_2'(s_0)}
  \quad \mbox{and} \quad
  \frac{(\bar r_1(s))^{n-1}}{\bar r_1'(s)} >
  \frac{(\bar r_2(s))^{n-1}}{\bar r_2'(s)}, \ s \in (s_0,M_1).
 \end{equation*}
 Since $\bar r_1(s)<\bar r_2(s)$ for $s \in (s_I,M_1]$, we have
 $\bar r_1(s_0)<\bar r_2(s_0)$, which shows that
 $0>\bar r_1'(s_0)>\bar r_2'(s_0)$.
 Since
 \begin{align*}
  \Biggl( \frac{(\bar r_1)^{n-1}}{\bar r_1'}
        - \frac{(\bar r_2)^{n-1}}{\bar r_2'} \Biggr)'(s)
  & = \Bigl( (\bar r_1)^{n-1} u'_1(\bar r_1)
            -(\bar r_2)^{n-1} u'_2(\bar r_2) \Bigr)' (s)\\
  & = -f(s) \Bigl( (\bar r_1(s))^{n-1} \bar r_1'(s)
                 - (\bar r_2(s))^{n-1} \bar r_2'(s) \Bigr) \\
  & = -f(s) \Bigl( \frac{(\bar r_1(s))^{n-1}}{\bar r_1'(s)}(\bar r_1'(s))^2
                - \frac{(\bar r_2(s))^{n-1}}{\bar r_2'(s)}(\bar r_2'(s))^2
            \Bigr),
 \end{align*}
 we have
 \begin{equation*}
  \Biggl( \frac{(\bar r_1)^{n-1}}{\bar r_1'}
 - \frac{(\bar r_2)^{n-1}}{\bar r_2'} \Biggr)'(s_0)
  = f(s_0) \frac{(\bar r_1(s_0))^{n-1}}{-\bar r_1'(s_0)}
    ( (\bar r_1'(s_0))^2 -(\bar r_2'(s_0))^2 )
  < 0,
 \end{equation*}
because $s_0 > s_I \geq \beta \ge B$ and $f(s_0)>0$.
 This is a contradiction.
 Therefore,
 \begin{equation*}
  \frac{(\bar r_1(s))^{n-1}}{\bar r_1'(s)} >
  \frac{(\bar r_2(s))^{n-1}}{\bar r_2'(s)}, \quad s \in (s_I,M_1).
 \end{equation*}
But for $s=s_I$ it is also true. Hence it is true for $s \in [s_I,M_1).$

 We note that condition $(f_3)$ implies that $P_i(s)$ is
 nonincreasing and $\bar P_i$ is nondecreasing.
 From Proposition \ref{subida2} we have $M_1<M_2$, and from Proposition
 \ref{lem31} we have $P_1(M_1)>P_2(M_1)$, so
 $$
  \bar P_1(M_1) =  P_1(M_1) > P_2(M_1) \ge P_2(M_2) = \bar P_2(M_2)
  \ge \bar P_2(M_1).
 $$
 By condition $(f_3)$, we have
 $$
  \left(\bar P_1-\bar P_2\right)'(s)
  = \Bigl( n-2-2n\Bigl(\frac{F}{f}\Bigr)'(s) \Bigr)
     \Bigl(\frac{(\bar r_1(s))^{n-1}}{\bar r_1'(s)}
         - \frac{(\bar r_2(s))^{n-1}}{\bar r_2'(s)} \Bigr)
  \le 0, \quad s \in (s_I,M_1),
 $$
 which implies that
 \begin{equation*}
  \bar P_1(s)-\bar P_2(s) \ge \bar P_1(M_1)-\bar P_2(M_1) > 0, \quad
  s \in [s_I,M_1].
 \end{equation*}
 Thus $\bar P_1(s)>\bar P_2(s)$ for $s\in[s_I,M_1]$.
\end{proof}

\begin{prop}\label{final1}{\bf (Step five)}
 Assume that $f$ satisfies $(f_1)$--$(f_3)$.
 Let $u_1$ and $u_2$ be positive solutions of \eqref{R} with
 $0<u_1'(a)<u_2'(a)$.
 If either
 \begin{enumerate}
  \item[(i)]  $\bar r_1(M_1)<\bar r_2(M_1)$ and $\bar r_1(s)$, $\bar r_2(s)$
 intersect each other in $(\beta,M_1)$,\quad or
  \item[(ii)] $\bar r_1(M_1)\geq \bar r_2(M_1)$,
 \end{enumerate}
 then
 \begin{equation}
  \bar r_1(\beta)>\bar r_2(\beta) \quad \mbox{and} \quad
  \frac{\bar r_1(\beta)}{\bar r_1'(\beta)} >
  \frac{\bar r_2(\beta)}{\bar r_2'(\beta)}.
   \label{at_beta}
 \end{equation}

\end{prop}

\begin{proof}
 (i)  Let $s_I$ be as in Proposition \ref{lem32}.
 Then we find that
 $\bar r_1'(s_I)<\bar r_2'(s_I)<0$.
 Since $\bar r_1(s_I)=\bar r_2(s_I)$, we have
 \begin{equation*}
  \frac{\bar r_1(s_I)}{\bar r_1'(s_I)} >
  \frac{\bar r_2(s_I)}{\bar r_2'(s_I)}.
 \end{equation*}
 Hence there exists $\delta>0$ such that
 \begin{equation}
  \bar r_1'(s)<\bar r_2'(s)<0, \quad
  \frac{\bar r_1(s)}{\bar r_1'(s)} >
  \frac{\bar r_2(s)}{\bar r_2'(s)}
   \label{barr1>barr2}
 \end{equation}
 for $s \in (s_I-\delta,s_I]$.

 First we assume that there exists $s_0 \in [\beta,s_I)$ such that
 \eqref{barr1>barr2} holds for $s \in (s_0,s_I]$ and
 $\bar r_1'(s_0)=\bar r_2'(s_0)$.
 The inequality $\bar r_1(s)/\bar r_1'(s)>\bar r_2(s)/\bar r_2'(s)$ is
 equivalent to
 \begin{equation*}
  (\log \bar r_1)'(s)< (\log \bar r_2)'(s).
 \end{equation*}
 Integrating this inequality over $(s_0,s_I)$ and using
 $\bar r_1(s_I)=\bar r_2(s_I)$, we have $\bar r_1(s_0)>\bar r_2(s_0)$.
 On the other hand, letting $s \to s_0$ in the inequality
 $\bar r_1(s)/\bar r_1'(s)>\bar r_2(s)/\bar r_2'(s)$ and using
 $\bar r_1'(s_0)=\bar r_2'(s_0)<0$, we have
 $\bar r_1(s_0) \le \bar r_2(s_0)$, which is a contradiction.
 Then $\bar r_1'(s)<\bar r_2'(s)<0$ for $s \in [\beta,s_I]$, which implies
 that $\bar r_1(s)>\bar r_2(s)$ for $s \in [\beta,s_I)$.

 Next we assume that there exists $s_0 \in [\beta,s_I)$ such that
 \eqref{barr1>barr2} holds for $s \in (s_0,s_I]$ and
 \begin{equation*}
  \frac{\bar r_1(s_0)}{\bar r_1'(s_0)} =
  \frac{\bar r_2(s_0)}{\bar r_2'(s_0)}.
 \end{equation*}
 Let
 $$
  A := \frac{(\bar r_2(s_0))^{n-1}\bar r_1'(s_0)}
            {(\bar r_1(s_0))^{n-1}\bar r_2'(s_0)}
     = \frac{(\bar r_2(s_0))^{n-2}}{(\bar r_1(s_0))^{n-2}}
     \le 1.
 $$
 Observe that
 $$
  \Bigl( \frac{(\bar r_2)^{n-1}\bar r_1'}
              {(\bar r_1)^{n-1}\bar r_2'} \Bigr)'(s)
   = f(s) \frac{(\bar r_2(s))^{n-1}\bar r_1'(s)}
               {(\bar r_1(s))^{n-1}\bar r_2'(s)}
                  ((\bar r_1'(s))^2-(\bar r_2'(s))^2)>0,
     \quad s \in (s_0,s_I].
 $$
 Hence,
 $$
  \frac{(\bar r_2(s))^{n-1}\bar r_1'(s)}{(\bar r_1(s))^{n-1}\bar r_2'(s)}
   > A, \quad s \in (s_0,s_I].
 $$
 As $\bar P_1(M_1)<0$ and $\bar P_1'(s) \ge 0$, we have $\bar P_1(s_I)<0$,
 which implies
 \begin{equation*}
  A\bar P_1(s_I)-\bar P_2(s_I) \ge \bar P_1(s_I)-\bar P_2(s_I) > 0,
 \end{equation*}
 by Proposition \ref{lem32}.
 By the direct computation we have
 $$
  (A\bar P_1-\bar P_2)' (s)
   = \Bigl(n-2-2n\Bigl(\frac{F}{f}\Bigr)'(s)\Bigr)
      \Bigl( A \frac{(\bar r_1(s))^{n-1}}{\bar r_1'(s)}
             - \frac{(\bar r_2(s))^{n-1}}{\bar r_2'(s)} \Bigr)
   \le 0
 $$
 for $s \in (s_0,s_I]$ and thus $A\bar P_1(s_0)-\bar P_2(s_0)>0$.
 On the other hand, since $\bar r_1(s_0)\ge\bar r_2(s_0)\ge0$, we have
 $$
  A\bar P_1(s_0)-\bar P_2(s_0)
  = -2F(s_0)(\bar r_2(s_0))^{n-2}((\bar r_1(s_0))^2-(\bar r_2(s_0))^2)
  \le 0,
 $$
 which is a contradiction.

 (ii)
 We note that
 \begin{equation*}
  \lim_{s\to M_1} \bar r_1'(s) = - \infty < \bar r_2'(M_1) < 0,
 \end{equation*}
 and hence
 \begin{equation*}
  \lim_{s\to M_1} \frac{\bar r_1(s)}{\bar r_1'(s)} = 0 >
  \frac{\bar r_2(M_1)}{\bar r_2'(M_1)}.
 \end{equation*}
 Therefore there exists $\delta>0$ such that \eqref{barr1>barr2} holds
 for $s \in (M_1-\delta,M_1)$.
 Arguing the same as (i), with $M_1$ instead of $s_I$,  we can conclude
 that \eqref{at_beta} holds.
\end{proof}

\begin{prop}\label{Intersec} {\bf(Step Six)}
Assume that $f$ satisfies $(f_1)$--$(f_4)$.
 Let $u_1$ and $u_2$ be positive solutions of \eqref{R} with $b<\infty$
 such that $0<u_1'(a)<u_2'(a)$.
 If $\bar r_1(M_1) < \bar r_2(M_1)$  then $\bar r_1(s)$ and $\bar r_2(s)$
 intersect in $(B,M_1)$.
\end{prop}
\begin{proof}

 Assume that $\bar r_1(M_1) < \bar r_2(M_1)$ and
 $\bar r_1(s)$ and $\bar r_2(s)$ do not intersect in $(B,M_1)$.
 First we suppose that $B>0$.
 From Propositions \ref{subida} and \ref{subida2} it follows that
 $r_1(s)>r_2(s)$ for $s \in (0,M_1]$.
 Then $u_2(r)>u_1(r) \ge B$ for all $r\in[r_1(B),\bar r_1(B))$ and
 $u_2(r_1(B))>u_1(r_1(B))=B$ and $u_2(\bar r_1(B))\ge u_1(\bar r_1(B))=B$.
 By multiplying the equations satisfied by $u_1$, $u_2$ by $u_2-B$ and
 $u_1-B$ respectively and subtracting, we obtain
 \begin{equation}\label{dif}
 \frac{d}{dr} \Bigl(r^{n-1}(u_1'(u_2-B)-u_2'(u_1-B))\Bigr)
  = -r^{n-1}(u_1-B)(u_2-B)\Bigl(\frac{f(u_1)}{u_1-B}
    -\frac{f(u_2)}{u_2-B}\Bigr).
 \end{equation}
 Since by $(f_4)$ the right hand side of \eqref{dif} is nonnegative in
 $[r_1(B),\bar r_1(B)]$, by integrating \eqref{dif} over this interval,
 we obtain the contradiction
 $$
  0 \ge  (\bar r_1(B))^{n-1}u_1'(\bar r_1(B))(u_2(\bar r_1(B))-B)
    \ge (r_1(B))^{n-1}u_1'(r_1(B))(u_2(r_1(B))-B)>0.
 $$
 Now we suppose that $B=0$.
 Proposition \ref{subida2} implies that $r_1(s)>r_2(s)$ for $s \in (0,M_1)$.
 Since $\bar r_1(M_1) < \bar r_2(M_1)$ and
 $\bar r_1(s)$ and $\bar r_2(s)$ do not intersect in $(B,M_1)$, we conclude
 that $u_2(s)>u_1(s)$ for $s \in (a,b)$.
 Integrating \eqref{dif} with $B=0$ over $[a,b]$, we have
 \begin{equation}
  0 = - \int_a^b r^{n-1}u_1(r)u_2(r)
      \Bigl(\frac{f(u_1(r))}{u_1(r)}-\frac{f(u_2(r))}{u_2(r)}\Bigr) dr.
 \end{equation}
 However, by $(f_4)$, the right hand side of this equality is positive.
 This is a contradiction.
\end{proof}

\begin{prop}\label{final2}{\bf (Final step)}
 If $f$ satisfies $(f_1)$--$(f_4)$, then there is at most one positive solution
 of \eqref{R} with $b<\infty$.
\end{prop}

\begin{proof}
i) First let $\beta=B =0$. Suppose $u_1$ and $u_2$ are two positive solutions of \eqref{R}
 with $0<u_1'(a)<u_2'(a)$. By Proposition \ref{Intersec} if $\bar r_1(M_1)<\bar r_2(M_1)$ there is an intersection point in $(B,M_1)$, so by
 Proposition \ref{final1} we have
\begin{equation}
  \bar r_1(\beta)>\bar r_2(\beta) \quad \mbox{and} \quad
  \frac{\bar r_1(\beta)}{\bar r_1'(\beta)} >
  \frac{\bar r_2(\beta)}{\bar r_2'(\beta)}.
  \end{equation}

 Then we have  $b=\bar r_1(\beta)>\bar r_2(\beta)=b$, a contradiction.
  So in this case we have the uniqueness.

ii) Suppose now that $\beta>B>0$
and  $u_1$ and $u_2$ are two positive solutions of \eqref{R}
 with $0<u_1'(a)<u_2'(a)$.

In order to arrive to a contradiction  we use the functionals
$$
 W_i(s)=\bar r_i(s)\sqrt{(u_i'(\bar r_i(s)))^2+2F(s)},\quad i=1,2,
$$
found in Franchi, Lanconelli and Serrin \cite{fls} or
Serrin and Tang \cite{st}.
By Proposition \eqref{1}, we note that $W_i(s)$, $i=1,2$ are well-defined.

By Proposition \ref{Intersec} if $\bar r_1(M_1)<\bar r_2(M_1)$ there is
an intersection point in $(B,M_1)$, and let $s_I$ be the largest intersection
point of $\bar r_1(s)$ and $\bar r_2(s)$ in $(B,M_1)$.
Let $s_0 \in(0,\beta]$ be defined as $\beta$ if either
$\bar r_1(M_1)\geq \bar r_2(M_1)$ or $\bar r_1(M_1)<\bar r_2(M_1)$ and
$s_I> \beta$.
Let $s_0=s_I$ if $\bar r_1(M_1)<\bar r_2(M_1)$ and $B< s_I\le \beta$.

In all three cases we have, by Proposition \ref {final1} or the definition of $s_I$, that $\bar r_1(s_0) \geq \bar r_2(s_0)$ and
 \begin{equation*}
       \frac{\bar r_1(s_0)}{\bar r_1'(s_0)}
  \geq \frac{\bar r_2(s_0)}{\bar r_2'(s_0)}.
 \end{equation*}
 By the uniqueness of solutions of initial value problem, we note that either
 $\bar r_1(s_0)\ne \bar r_2(s_0)$ or $\bar r_1'(s_0)\ne \bar r_2'(s_0)$.
 Then $W_1(s_0)<W_2(s_0)$ and $|u_1'(\bar r_1(s_0))|<|u_2'(\bar r_2(s_0))|$.
 Also,
 \begin{align*}
  (&W_1 -W_2)'(s)\\
   & = (n-2) \Bigl(\frac{|u_1'(\bar r_1(s))|}{\sqrt{(u_1'(\bar r_1(s)))^2+2F(s)}}
     - \frac{|u_2'(\bar r_2(s))|}{\sqrt{(u_2'(\bar r_2(s)))^2+2F(s)}}\Bigr) \\
    & \quad -2F(s)\Bigl(\frac{1}{|u_1'(\bar r_1(s))|\sqrt{(u_1'(\bar r_1(s)))^2+2F(s)}}
     - \frac{1}{|u_2'(\bar r_2(s))|\sqrt{(u_2'(\bar r_2(s)))^2+2F(s)}}\Bigr).
 \end{align*}
 Since $F(s)<0$ for $s\in(0,s_0)$, we note that, for each fixed
 $s\in(0,s_0)$, the functions
 $$
  x \mapsto \frac{x}{\sqrt{x^2+2F(s)}} \quad \mbox{and} \quad
  x \mapsto \frac{1}{x\sqrt{x^2+2F(s)}}
 $$
 are decreasing for $x>\sqrt{-2F(s)}$.
 Therefore, as long as $|u_1'(\bar r_1(s))|<|u_2'(\bar r_2(s))|$ and $F(s)\le0$,
 we have $(W_1-W_2)'(s)>0$ and $(\bar r_1-\bar r_2)'(s)<0$.
 Assume that there exists $s_1\in(0,s_0)$ such that
 $|u_1'(\bar r_1(s_1))|=|u_2'(\bar r_2(s_1))|$ and
 $|u_1'(\bar r_1(s))|<|u_2'(\bar r_2(s))|$ for $s\in(s_1,s_0]$.
 Then $W_1(s_1)<W_2(s_1)$ and hence $\bar r_1(s_1)<\bar r_2(s_1)$.
 On the other hand, since $(\bar r_1(s)- \bar r_2(s))'<0$ for
 $s\in(s_1,s_0]$, we find that
 \begin{equation*}
  \bar r_1(s_1)-\bar r_2(s_1) > \bar r_1(s_0)-\bar r_2(s_0) \ge 0,
 \end{equation*}
 which is a contradiction.
 Consequently, $|u_1'(\bar r_1(s))|<|u_2'(\bar r_2(s))|$ for $s\in(0,s_0]$.

Hence $W_1(s)-W_2(s)<W_1(s_0)-W_2(s_0)<0$ and $\bar r_1(s)-\bar r_2(s) >\bar r_1(s_0)-\bar r_2(s_0) $ in $[0,s_0)$, from where we obtain the contradiction
  \begin{equation*}
  0 = b - b
    = \bar r_1(0)-\bar r_2(0)
    > \bar r_1(s_0)-\bar r_2(s_0) \ge 0.
 \end{equation*}
 This finishes the proof of Theorem \ref{main0}.

\end{proof}

 \section{Proof of Theorem \ref{2}}

 To prove Theorem \ref{2}, we will follow the same steps as in Theorem \ref{main0}. We assume that \eqref{R}, now with $b=\infty$ and hence $B>0$,  has two distinct
positive solutions $u_1$ and $u_2$ with  $0<u_1'(a)<u_2'(a)$ and arrive at a contradiction.

We will assume only that  $f$ satisfies $(f_1)$--$(f_3)$ for all steps, as in this case we do not need $(f_4)$ for the existence of a
intersection point.

The proof of Step One, Step Two, Step Three, Step Four and Step Five are exactly the same as in Theorem \ref{main0} if you ignore the comments in Step One and  Step Three for the case $B=0$.

\begin{prop}\label{Intersec2} {\bf(Step Six)}
 Assume that $f$ satisfies $(f_1)$--$(f_3)$.
 Let $u_1$ and $u_2$ be positive solutions of \eqref{R} with $b=\infty$,
 and $B>0$ such that $0<u_1'(a)<u_2'(a)$.
 If $\bar r_1(M_1) < \bar r_2(M_1)$  then $\bar r_1(s)$ and $\bar r_2(s)$
 intersect in $(0,M_1)$.
\end{prop}
\begin{proof}
Assume that $\bar r_1(s)<\bar r_2(s)$ for all $s\in(0,M_1)$.
 We will first prove that $u_1'(\bar r_1(s))>u_2'(\bar r_2(s))$ in $(0,M_1]$.
 Indeed, this is true at $M_1$ by assumption, and assume by
 contradiction that there exists the largest $s_0\in(0,M_1)$ such that
 $u_1'(\bar r_1(s_0))=u_2'(\bar r_2(s_0))$.
 Then,
 \begin{eqnarray*}
  \frac{d}{ds} \Bigl(u_1'(\bar r_1(s))-u_2'(\bar r_2(s))\Bigr) \Bigm|_{s=s_0}
    = \frac{u_1''(\bar r_1(s_0))}{u_1'(\bar r_1(s_0))}
      -\frac{u_2''(\bar r_2(s_0))}{u_2'(\bar r_2(s_0))} \\
    = - (n-1)\Bigl(\frac{1}{\bar r_1(s_0)}-\frac{1}{\bar r_2(s_0)}\Bigr)
       - f(s_0)\Bigl( \frac{1}{u_1'(\bar r_1(s_0))}
                    - \frac{1}{u_2'(\bar r_2(s_0))} \Bigr) \\
    = - (n-1)\Bigl(\frac{1}{\bar r_1(s_0)}-\frac{1}{\bar r_2(s_0)}\Bigr)<0,
 \end{eqnarray*}
 which cannot be. Hence $u_1'(\bar r_1(s))>u_2'(\bar r_2(s))$ in $(0,M_1]$.
 In what follows we use an idea in \cite{st} and consider the functional
 $J(s):=|u_1'(\bar r_1(s))|^2-|u_2'(\bar r_2(s))|^2$.
 Then, for each $s\in(0,M_1)$, we have $J(s)<0$ and
 \begin{align*}
  J'(s) & = 2(u_1''(\bar r_1(s))-u_2''(\bar r_2(s))) \\
   & = 2(n-1)\Bigl(\frac{|u_1'(\bar r_1(s))|}{\bar r_1(s)}
     - \frac{|u_2'(\bar r_2(s))|}{\bar r_2(s)}\Bigr)\\
   & = \frac{2(n-1)}{|u_1'(\bar r_1(s))|}
        \Bigl(\frac{|u_1'(\bar r_1(s))|^2}{\bar r_1(s)}
     - \frac{|u_1'(\bar r_1(s))||u_2'(\bar r_2(s))|}{\bar r_2(s)}\Bigr) \\
   & > \frac{2(n-1)}{|u_1'(\bar r_1(s))|}
        \Bigl(\frac{|u_1'(\bar r_1(s))|^2}{\bar r_1(s)}
     - \frac{|u_1'(\bar r_1(s))|^2+|u_2'(\bar r_2(s))|^2}{2\bar r_2(s)}\Bigr) \\
   & > \frac{2(n-1)}{|u_1'(\bar r_1(s))|\bar r_1(s)}\Bigl(|u_1'(\bar r_1(s))|^2
     - \frac{|u_1'(\bar r_1(s))|^2+|u_2'(\bar r_2(s))|^2}{2}\Bigr) \\
   & = \frac{n-1}{|u_1'(\bar r_1(s))|\bar r_1(s)}J(s),
 \end{align*}
 that is,
 $$
  \frac{J'(s)}{J(s)}<-(n-1)\frac{\bar r_1'(s)}{\bar r_1(s)}.
 $$
 Let now $\varepsilon\in(0,M_1/2)$.
 By integrating this inequality over $(\varepsilon,M_1/2)$, we obtain
 \begin{align*}
  (\bar r_1(M_1/2))^{n-1}J(M_1/2)
   & > (\bar r_1(\varepsilon))^{n-1}J(\varepsilon) \\
   & = (\bar r_1(\varepsilon))^{n-1}(|u_1'(\bar r_1(\varepsilon))|^2
     - |u_2'(\bar r_2(\varepsilon))|^2) \\
   & > (\bar r_1(\varepsilon))^{n-1}|u_1'(\bar r_1(\varepsilon))|^2
     - (\bar r_2(\varepsilon))^{n-1}|u_2'(\bar r_2(\varepsilon))|^2 \\
   & = \frac{((\bar r_1(\varepsilon))^{n-1}u_1'(\bar r_1(\varepsilon)))^2}
            {(\bar r_1(\varepsilon))^{n-1}}
     - \frac{((\bar r_2(\varepsilon))^{n-1}u_2'(\bar r_2(\varepsilon)))^2}
            {(\bar r_2(\varepsilon))^{n-1}}.
 \end{align*}
 By Proposition \ref{u'->0}, there exist $L_i\in (-\infty,0]$, $i=1$, $2$
 such that $r^{n-1}u_i'(r)\to L_i$ as $r\to\infty$.
 Therefore,
 $$
  \lim_{\varepsilon\to0}
   \frac{((\bar r_i(\varepsilon))^{n-1}u_i'(\bar r_i(\varepsilon)))^2}
        {(\bar r_i(\varepsilon))^{n-1}} = 0, \quad i=1,2.
 $$
 and we obtain the contradiction $0>(\bar r_1(M_1/2))^{n-1}J(M_1/2)\ge 0$.
\end{proof}

\begin{prop}\label{final22}{\bf (Final step)}
	
	If $f$ satisfies $(f_1)$-$(f_3)$ there  is at most one positive solutions of \eqref{R} with $b=\infty.$
	
\end{prop}
\begin{proof}

Suppose   $u_1$ and $u_2$ are two positive solutions of \eqref{R} with $b=\infty$
and  $0<u_1'(a)<u_2'(a)$.

We use, as in Theorem \ref{main0},  the functionals
$$
 W_i(s)=\bar r_i(s)\sqrt{(u_i'(\bar r_i(s)))^2+2F(s)},\quad i=1,2.
$$
By Proposition \ref{Intersec2} if $\bar r_1(M_1)<\bar r_2(M_1)$ there is
an intersection point in $(0,M_1)$, and let $s_I$ be the largest intersection
point of $\bar r_1(s)$ and $\bar r_2(s)$ in $(B,M_1)$.
Let $s_0 \in(0,\beta]$ be defined as $\beta$ if either
$\bar r_1(M_1)\geq \bar r_2(M_1)$ or  $\bar r_1(M_1)<\bar r_2(M_1)$ and
$s_I> \beta$.
Let $s_0=s_I$ if $\bar r_1(M_1)<\bar r_2(M_1)$ and $0< s_I\le \beta$.

In all three cases we have, by Proposition \ref {final1} or the definition of $s_I$, that $\bar r_1(s_0) \geq \bar r_2(s_0)$ and
 \begin{equation*}
       \frac{\bar r_1(s_0)}{\bar r_1'(s_0)}
  \geq \frac{\bar r_2(s_0)}{\bar r_2'(s_0)}.
 \end{equation*}
 By the uniqueness of solutions of initial value problem, we note that either
 $\bar r_1(s_0)\ne \bar r_2(s_0)$ or $\bar r_1'(s_0)\ne \bar r_2'(s_0)$.
 Then $W_1(s_0)<W_2(s_0)$ and $|u_1'(\bar r_1(s_0))|<|u_2'(\bar r_2(s_0))|$.

 As in the proof of Theorem \ref{main0} we can now prove that
 $W_1(s)-W_2(s)<W_1(s_0)-W_2(s_0)<0$ and $\bar r_1(s)-\bar r_2(s) >\bar r_1(s_0)-\bar r_2(s_0) $ in $(0,s_0)$.

 By Proposition \ref{1} and $(f_2)$, we observe that
 \begin{align*}
 0 < W_1(s)
 & < W_2(s) + W_1(s_0)-W_2(s_0) \\
 & < \bar r_2(s)\sqrt{(u_2'(\bar r_2(s)))^2} + W_1(s_0)-W_2(s_0) \\
 & = \bar r_2(s) |(u_2'(\bar r_2(s))| + W_1(s_0)-W_2(s_0)
 \end{align*}
 Observe that  $\bar r_i(s) \to \infty$ as $s \to 0$. Letting $s \to 0$, by               Proposition             \ref{u'->0},  we have $0 \le W_1(s_0)-W_2(s_0) < 0$.  A contradiction.

This finishes the proof of Final step and hence of Theorem \ref{2}.
\end{proof}
\section{Proof of the Corollaries}
\label{proofofcor}

In this section we give a proof of Corollaries \ref{cor1} and \ref{cor2}.

\begin{proof}[\it Proof of Corollary \ref{cor1}]
 Assume that $f(s)=s^p+s^q$ for $s \ge 0$ and $1\le q<p$.
 Then $f \in C^1[0,\infty)$ and $(f_1)$, $(f_2)$ and $(f_4)$ hold with $B=0$.
 By Theorem \ref{main0}, it is enough to prove that $(f_3)$ with $\beta=0$
 is satisfied when one of the (i)--(iii) of Corollary \ref{cor1} holds.
 Since
 \begin{equation*}
  \left( \frac{F}{f} \right)'(s) = 1 - \frac{F(s)f'(s)}{(f(s))^2},
 \end{equation*}
 condition $(f_3)$ with $\beta=0$ holds, if
 \begin{equation}
  g(s):=\frac{F(s)f'(s)}{(f(s))^2} \le \frac{n+2}{2n}, \quad s>0.
   \label{g>(n+2)/(2n)}
 \end{equation}
 Hence we will show \eqref{g>(n+2)/(2n)}.
 Observe that
 \begin{equation*}
  g(s) = \frac{p}{p+1} + \frac{p-q}{(p+1)(q+1)} h(s^{p-q}+1),
 \end{equation*}
 where
 \begin{equation*}
  h(t) = \frac{p-q-1}{t} - \frac{p-q}{t^2}, \quad t\ge 1.
 \end{equation*}

 Now we show that \eqref{g>(n+2)/(2n)} holds, if the following (a) or (b) holds:
 \begin{enumerate}
  \item[(a)] $p\le q+1$ and $(n-2)p\le (n+2)$;
  \item[(b)] $p>q+1$ and $\frac{(p+q+1)^2}{2(p+1)(q+1)} \le \frac{n+2}{n}$.

 \end{enumerate}

 First we assume (a).
 Then $h(t)$ is increasing and hence
 \begin{equation*}
  g(s) < \frac{p}{p+1}, \quad s>0.
 \end{equation*}
 If $n>2$, then $p\le (n+2)/(n-2)$ implies $p/(p+1)\le(n+2)/(2n)$.
 If $n=2$, then $p/(p+1)<1=(n+2)/(2n)$ for $p>1$.
 Therefore, \eqref{g>(n+2)/(2n)} holds.

 Next we suppose (b).
 Then we see that $2(p-q)/(p-q-1)>1$ and
 \begin{equation*}
  h(t) \le h\left( \frac{2(p-q)}{(p-q-1)} \right)
  = \frac{(p-q-1)^2}{4(p-q)}, \quad t \ge 1.
 \end{equation*}
 Hence
 \begin{equation*}
  g(s) \le \frac{p}{p+1} + \frac{(p-q-1)^2}{4(p+1)(q+1)}
  = \frac{(p+q+1)^2}{4(p+1)(q+1)}, \quad s>0,
 \end{equation*}
 which implies \eqref{g>(n+2)/(2n)}.

 Finally we check (a) or (b) holds when one of (i)--(iv) is satisfied.
 If (i) holds, then $(n+2)/(n-2)\le2$ and hence
 $p\le 2 \le q+1$, which shows (a) is satisfied.
 If (ii) holds, then
 \begin{equation*}
  p \le \frac{n+2}{n-2} = \frac{4}{n-2} + 1 \le q + 1
 \end{equation*}
 and hence (a) holds.

 Now we assume (iii).
 We claim that $p\le(n+2)/(n-2)$.
 Since
 \begin{equation*}
  n^2(q+1)^2 - \frac{2n^2(n+2)}{n-2}(q+1) + \frac{n^2(n+2)^2}{(n-2)^2}
  = n^2\left( (q+1) - \frac{n+2}{n-2} \right)^2 \ge 0,
 \end{equation*}
 we have
 \begin{multline*}
  4(q+1)^2 - \frac{4n(n+2)}{n-2}(q+1) + \frac{n^2(n+2)^2}{(n-2)^2}  \\
    \ge (4-n^2)(q+1)^2
      + \left( \frac{2n^2(n+2)}{n-2} -\frac{4n(n+2)}{n-2} \right) (q+1),
 \end{multline*}
 which is equivalent to
 \begin{equation*}
  \left( 2(q+1) - \frac{n(n+2)}{n-2} \right)^2 \ge (n+2)(q+1)(n+2-(n-2)q).
 \end{equation*}
 Since $n>2$ and $q<4/(n-2)$, we note that
 \begin{equation*}
  2(q+1) - \frac{n(n+2)}{n-2} < 0, \quad n+2-(n-2)q>0.
 \end{equation*}
 Then we have
 \begin{equation*}
  \frac{n(n+2)}{n-2} - 2(q+1) \ge \sqrt{(n+2)(q+1)(n+2-(n-2)q)},
 \end{equation*}
 which shows $p\le (n+2)/(n-2)$.
 Hence, if $p\le q+1$, then (a) is satisfied.
 We suppose that $p>q+1$.
 Then we note that $np - 2(q+1)>2p-2(q+1)>0$.
 From the last inequality of (iii), it follows that
 \begin{equation*}
  \sqrt{(n+2)(q+1)(n+2-(n-2)q)} \ge np - 2(q+1),
 \end{equation*}
 which implies that
 \begin{equation*}
  (n+2)(q+1)(2n-(n-2)(q+1)) \ge (np-2(q+1))^2.
 \end{equation*}
 This is equivalent to the second inequality in (b).
 Then we have (b).

 We assume (iv).
 If $p\le q+1$, then (a) holds.
 If $p>q+1$, then $0<p-(q+1) \le 2\sqrt{q+1}$,
 and hence $(p-(q+1))^2\le 4(q+1)$, which is equivalent to
 $(p+q+1)^2/((p+1)(q+1))\le 4$.
 Therefore, (b) is satisfied.
\end{proof}
\begin{proof}[\it Proof of Corollary \ref{cor2}]
	
It can be easily verified that in this case
\begin{equation}\label{sub}\Bigl(\frac{F}{f}\Bigr)'(s)=\frac{1}{p+1}+\frac{(p-q)(p-q-1)}{(p+1)(q+1)}\frac{1}{t-1}+\frac{(p-q)^2}{(p+1)(q+1)}\frac{1}{(t-1)^2}
\end{equation}
where $t=s^{p-q}$. Hence  if  $p\ge q+1$, we have that
$$\Bigl(\frac{F}{f}\Bigr)'(s)\ge \frac{1}{p+1}$$
and therefore $(f_3)$ is satisfied if either
\begin{equation}\label{q+1<p<(n+2)/(n-2)}
 n>2, \quad q+1\le p\le\frac{n+2}{n-2}
\end{equation}
or
$$n=2, \quad q+1 \le p.$$
If $q+1>p>q$ the minimum of $ \Bigl(\frac{F}{f}\Bigr)'(s)$ is at
$(t-1) =\frac{2(p-q)}{(1+p-q)}$ and so
\begin{equation}\label{min2}
\Bigl(\frac{F}{f}\Bigr)'(s)\geq\frac{1}{(p+1)}(1-\frac{(q+1-p)^2}{4(q+1)}).
\end{equation}
Note that $q+1-p=1-(p-q)<1$ and the right hand side in \eqref{min2}
is positive and we conclude that $(f_3)$ is satisfied if either
$$n>2, \quad q<p<q+1, \quad
  p \le \frac{\Bigl(1-\frac{(q+1-p)^2}{2(q+1)}\Bigr)n+2}{n-2}$$
or
$$n=2, \quad q<p<q+1.$$
Consequently, if (iii) is satisfied, then $(f_3)$ holds.

In the case $n>2$, we note that
\begin{equation*}
 p \le \frac{\Bigl(1-\frac{(q+1-p)^2}{2(q+1)}\Bigr)n+2}{n-2}
\end{equation*}
is equivalent to $q \le (n+2)/(n-2)$ and $P_-(q)\le p \le P(q)$, where
\begin{equation*}
 P_-(q):=\frac{2(q+1)-\sqrt{(n+2)(q+1)(n+2-(n-2)q)}}{n}.
\end{equation*}
Moreover, $P(q)\ge q+1$ is equivalent to $q\le 4/(n-2)$.
We find that $P_-(q)\le q$ for $0<q<(n+2)/(n-2)$.
Hence, $(f_3)$ holds if one of \eqref{q+1<p<(n+2)/(n-2)} and the following
is satisfied:
\begin{gather}
 n>2, \quad q\le 4/(n-2), \quad q<p<q+1; \label{q<4/(n-2)} \\
 n>2, \quad 4/(n-2)<q\le(n+2)/(n-2), \quad q<p\le P(q). \label{q>4/(n-2)}
\end{gather}
From \eqref{q+1<p<(n+2)/(n-2)} and \eqref{q<4/(n-2)}, it follows that
(i) implies $(f_3)$.
Since
\begin{equation*}
 q<P(q) \quad \mbox{for} \ 0<q<\frac{4+\sqrt{2n(n+2)}}{2(n-2)},
\end{equation*}
by \eqref{q>4/(n-2)}, we conclude that (ii) implies $(f_3)$.

As for the superlinear assumption $(f_4)$ that is needed when $b<\infty$, it can be easily verified that it is satisfied if $p>1$ and $0<q<p$. Indeed, this condition is satisfied provided $g(s)>0$ for all $s>1$, where
$$g(s)=(p-1)s^{p-q+1}-ps^{p-q}-(q-1)s+q,$$
from where the result follows.

\end{proof}

\end{document}